\documentclass[10pt,notitlepage,twoside,a4paper]{amsart}
 \usepackage{amsfonts}

\usepackage{amsmath,amssymb,enumerate}

\usepackage{epsfig,fancyhdr,color}

\usepackage{amssymb}
\usepackage{amsmath,amsthm}
\usepackage{latexsym}
\usepackage{amscd}
\usepackage{psfrag}
\usepackage{graphicx}
\usepackage[latin1]{inputenc}
\usepackage[mathcal]{eucal}

\renewcommand{\textsc}{\textcolor{red}}

\newtheorem{theorem}{\rm\bf Theorem}[section]
\newtheorem{proposition}[theorem]{\rm\bf Proposition}
\newtheorem{lemma}[theorem]{\rm\bf Lemma}
\newtheorem{corollary}[theorem]{\rm\bf Corollary}
\newtheorem*{theorem 1}{\rm\bf Proposition 1}
\newtheorem*{theorem 2}{\rm\bf Proposition 2}

\theoremstyle{definition}
\newtheorem{definition}[theorem]{\rm\bf Definition}

\theoremstyle{remark}
\newtheorem{remark}[theorem]{\rm\bf Remark}

\newtheorem{example}[theorem]{\rm\bf Example}

\def\interieur#1{\mathord{\mathop{\kern 0pt #1}\limits^\circ}}

\title[Teichm\"uller spaces of surfaces with boundary]{On length spectrum metrics and weak metrics on Teichm\"uller spaces of surfaces with boundary}

\author{Lixin Liu}
\address{Lixin Liu, Department of Mathematics, Zhongshan University, 510275, Guangzhou, P. R. China}
\email{mcsllx@mail.sysu.edu.cn}

\author{Athanase Papadopoulos}
\address{Athanase Papadopoulos, Max-Plank-Institut f\"ur Mathematik, Vivatsgasse 7, 53111 Bonn, Germany ;  Institut de Recherche Math{\'e}matique Avanc\'ee,
Universit{\'e} de Strasbourg and CNRS,
7 rue Ren\'e Descartes,
 67084 Strasbourg Cedex, France} \email{papadopoulos@math.u-strasbg.fr}
\date{\today}

\author{Weixu Su}
\address{Weixu SU, Department of Mathematics, Zhongshan University, 510275, Guangzhou, P. R. China}
\email{su023411040@163.com}

\author{Guillaume Th\'eret}
\address{Guillaume Th\'eret, Max-Plank-Institut f\"ur Mathematik, Vivatsgasse 7, 53111 Bonn, Germany}
\email{theret@mpim-bonn.mpg.de}

\begin{document}

\begin{abstract}    

We define and study metrics and weak metrics on the Teichm\"uller space of a surface   of topologically finite type with boundary. 
These metrics and weak metrics are associated to the hyperbolic length spectrum of simple closed curves and of properly embedded arcs in the surface. We give a comparison between the defined metrics on regions of Teichm\"uller space which we call    {\it $\varepsilon_0$-relative $\epsilon$-thick parts}, for $\epsilon >0$ and $\varepsilon_0\geq \epsilon>0$.

\bigskip

\noindent AMS Mathematics Subject Classification:   32G15 ; 30F30 ; 30F60.
\medskip

\noindent Keywords: Teichm\"uller space, length spectrum metric, length spectrum weak metric, Thurston's asymmetric metric.
\medskip

\noindent  Lixin Liu was partially supported by NSFC (No.
10871211).
 
\end{abstract}
\maketitle
\tableofcontents

\section{Introduction}
\label{intro}

In this paper, $S$ is a connected oriented surface of finite topological type whose  boundary is nonempty unless specifically specified. 
More precisely, $S$ is obtained from a closed surface of genus $g\geq 0$ by removing a finite number $p\geq 0$ of punctures 
and a finite number $b\geq 1$ of disjoint open disks. 
We shall say that $S$ has $p$ punctures and $b$ boundary components. 
The Euler characteristic of $S$ is equal to $\chi(S)= 2-2g-p-b$, and we assume throughout the paper that $\chi(S)<0$.
The boundary of $S$ is denoted by $\partial S$ and by assumption we have $\partial S\not=\emptyset$.

We shall equip $S$ with complete hyperbolic structures of finite area with totally geodesic boundary, 
and by this we mean that the following two properties are satisfied:
\begin{enumerate}
\item each puncture has a neighborhood which is isometric to a cusp, i.e., 
the quotient of $\{z=x+iy \in\mathbb{H}^{2}\  | \ a < y \}$, for some $a>0$, by the group generated by the translation $z\mapsto z+1$ ;
\item each boundary component is a smooth simple closed geodesic. 
\end{enumerate}

Let $S^d= S\cup \bar{S}$ denotes the double of $S$,  obtained by taking a copy, $\bar{S}$, of $S$ and by identifying the corresponding boundary components by 
an orientation-reversing homeomorphism.
The surface $S^d$ carries a canonical orientation-reversing involution, whose fixed point set is the boundary $\partial S$ of $S$, considered as embedded in $S^d$.
The oriented closed surface $S^d$ has genus $2g+b-1$ and $2p$ punctures.

If $S$ is equipped with a hyperbolic structure, we shall often equip $S^d$ with the \textsl{doubled hyperbolic structure}, that is, 
with the unique hyperbolic structure on $S^{d}$ that restricts to the structure on $S\subset S^d$ we started with,
and that makes the canonical involution of $S^d$ an isometry.

We denote by $\mathcal {T}(S)$ the reduced Teichm\"{u}ller space of marked hyperbolic structures on $S$. 
Recall that $\mathcal {T}(S)$ is the set of equivalence classes of pairs $(X, f)$, where $X$ is a hyperbolic surface (of the type we consider here) and
$f : S\rightarrow X$ is a homeomorphism (called the marking), and where $(X_1, f_1)$ is said to be equivalent to $(X_2, f_2)$ if there is an isometry 
$h : X_1\rightarrow X_2$ which is homotopic to $f_2\circ f_1^{-1}$. 
We recall that in this reduced theory, homotopies need not fix the boundary of $S$ pointwise. 
Since all Teichm\"uller spaces that we consider are reduced, we shall omit the word ``reduced" in our exposition. 
Furthermore, we shall denote an element $(X, f)$ of $\mathcal {T}(S)$ by $X$, without explicit reference to the marking.
We shall also denote any representative of an element $X$ of $\mathcal {T}(S)$ by the same letter, if no confusion arises.
If $X$ is a hyperbolic structure on $S$, we shall denote by $X^d$ the doubled structure on $S^d$.
 
 A simple closed curve on $S$ is said to be {\it peripheral} if it is homotopic to a puncture. 
It is said to be {\it essential} if it is not peripheral, and if it is not homotopic to a point (but it can be homotopic to a boundary component).

We let $\mathcal{C}=\mathcal{C}(S)$ be the set of homotopy classes of essential simple closed curves on $S$. 

An {\it arc} in $S$ is the homeomorphic image of a closed interval which is properly embedded in $S$ 
(that is, the interior of the arc is in the interior of $S$ and endpoints of the arc are on the boundary of $S$). 
All homotopies of arcs that we consider are relative to $\partial S$, that is, they leave the endpoints of arcs on the set $\partial S$ 
(but they do not necessarily fix pointwise the points on the boundary).
An arc is said to be {\it essential} if it is not homotopic (relative to $\partial S$) to a subset of $\partial S$. 

We let $\mathcal{B}=\mathcal{B}(S)$ be the set of homotopy classes
of essential arcs on $S$ union the set of homotopy classes of simple closed curves which are homotopic to boundary components. 
Note that $\mathcal{B}\cap\mathcal{C}$ consists in all components of the boundary of $S$.

For any $\gamma\in \mathcal{B}\cup\mathcal{C}$ and for any hyperbolic structure $X$, we let $\gamma^X$ be the geodesic representative of $\gamma$ 
(that is, the curve of shortest length in the homotopy class relative to $\partial S$). 
The geodesic $\gamma^X$ is unique, and it is orthogonal to $\partial S$ at each intersection point, in the case where $\gamma$ is an equivalence class of an arc. 
We denote by $l_X(\gamma)$ the length of $\gamma^X$ with respect to the hyperbolic metric considered. This
  length only depends upon the class of $X$ in  $\mathcal{T}(S)$.

We denote by $\mathcal{ML}(S)$ the space of measured geodesic laminations on $S$, whenever $S$ is equipped with a hyperbolic metric. 
This space is equipped with the topology defined by Thurston (cf. \cite{Thurston-notes}). 
We also recall that there are natural homeomorphisms between the various spaces $\mathcal{ML}(S)$ when the hyperbolic structure on $S$ varies, 
so that it is possible to talk about a measured geodesic lamination on $S$ without referring to a specific hyperbolic structure on the surface.

A  measured lamination (respectively, hyperbolic structure, simple closed curve, etc.) on $S^d$ is said to be {\it symmetric} if it is invariant by the canonical involution.\\

We conclude these preliminaries by recalling two trigonometric formulae from hyperbolic geometry that will be useful in the sequel.

The first useful formula concerns hyperbolic right-angled hexagons (that is, hexagons in the hyperbolic plane whose angles are right angles). 
Let $a,c',b,a',c,b'$ be the lengths of the consecutive edges of such a hexagon, cf. Figure \ref{hexagon2}.  
Then, we have
\begin{equation}\label{eq:hexagon1}
\cosh a = -\cosh b\cosh c+\sinh b\sinh c\cosh a'.
\end{equation}
This formula allows to express $a'$ in terms of $a,b,c$, and it implies in particular that the isometry type of a right-angled hexagon 
is determined by the length of any three non-consecutive edges. 

\begin{figure}[!hbp]
\psfrag{a}{$a$}
\psfrag{b}{$b$}
\psfrag{c}{$c$}
\psfrag{a'}{$a'$}
\psfrag{b'}{$b'$}
\psfrag{c'}{$c'$}
\centering
\includegraphics[width=.4\linewidth]{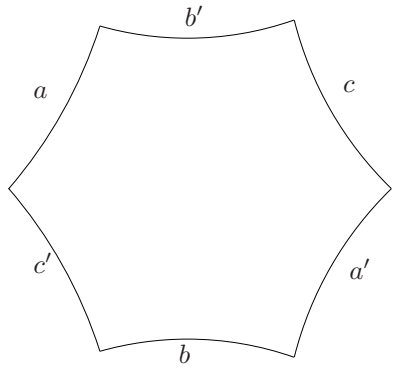}
\caption{\small{A right-angled hexagon in hyperbolic space.}}
\label{hexagon2}
\end{figure}

The second useful formula gives a relation between the lengths of edges of a right-angled pentagon in the hyperbolic plane.

We consider a pentagon with five right angles, with consecutive edges of lengths  $a,c',b,c,b'$, as in Figure \ref{eq:pentagon1}. Then, we have
\begin{equation}\label{eq:pentagon1}
\cosh a = \sinh b\sinh c.
\end{equation}

\begin{figure}[!hbp]
\psfrag{a}{$a$}
\psfrag{b}{$b$}
\psfrag{c}{$c$}
\psfrag{b'}{$b'$}
\psfrag{c'}{$c'$}
\centering
\includegraphics[width=.4\linewidth]{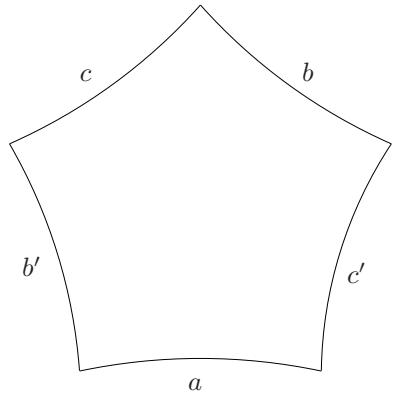}
\caption{\small{A right-angled pentagon in hyperbolic space.}}
\label{pentagon1}
\end{figure}

For the proofs of these formulae, we refer to \cite{Fenchel} pp. 85 and 86.

The goal of this paper is to introduce and study some metrics and weak metrics on the Teichm\"uller space $\mathcal {T}(S)$. 
These metrics and weak metrics are defined using  the hyperbolic length spectrum of simple closed curves and of properly embedded arcs in $S$.

Considering surfaces with boundary is important for several reasons, one of them being that there are inclusion maps between surfaces with boundary that induce maps between various associated spaces of geometric strutures on these surfaces (Teichm\"uller spaces, measured laminations spaces, etc.), and of course this phenomenon does not occur if we restrict the theory to surfaces without boundary.

One can introduce on surfaces with boundary objects that are analogous to objects on surfaces without boundary, for instance, analogs of Thurston's asymmetric weak metrics,  studied in \cite{Thurston1998}, but it turns out that there are interesting differences that we think are worth studying carefully, and in some sense this is what we do in this paper.

\section{Weak metrics on $\mathcal{T}(S)$}
\label{section:weak}

\begin{definition}[Weak metric]
A {\it weak metric} is a structure
satisfying the axioms of a metric space except possibly the symmetry axiom. 
In other words, a weak metric on a set $M$ is a function
$\delta:M\times M\to [0,\infty)$   satisfying
\begin{enumerate}[(a)]
 \item $\delta(x,y)=0\iff x=y$ for all $x$ an $y$ in $M$;
\item  $\delta(x,y)+\delta(y,z)\geq \delta(x,z)$ for
all $x$, $y$ and $z$ in $M$.
\end{enumerate}
We shall say that the weak metric $\delta$ is {\it asymmetric} if the following holds:
 \medskip
 
\noindent (c)  there exist $x$ and $y$ in $M$ satisfying 
$\delta(x,y)\not=\delta(y,x)$.
\end{definition}

We consider the following two functions on $\mathcal{T}(S)\times \mathcal{T}(S)$:

\begin{equation}
\label{eq:def1} 
d(X,Y) = \log \sup_{\alpha \in  \mathcal {B}\cup  \mathcal {C}}\frac{l_Y(\alpha)}{l_X(\alpha)}
\end{equation}

\begin{equation}
\label{eq:def2}
\overline{d}(X,Y) = \log \sup_{\alpha \in \mathcal {B}\cup  \mathcal {C}}\frac{l_X(\alpha)}{l_Y(\alpha)}.
\end{equation}

(Note that $\overline{d}(X,Y)=d(Y,X)$, for all $X,Y\in\mathcal{T}(S)$.)

These two functions are analogues, for surfaces with boundary, of asymmetric weak metrics introduced by Thurston in \cite{Thurston1998} for surfaces of finite type without boundary. 
The introduction of the set $\mathcal{B}$ in the definitions of $d$ and $\overline{d}$ is natural when dealing with surfaces with nonempty boundary, 
and it turns out to be essential in what follows.

\subsection{The functions $d$ and $\overline{d}$ are asymmetric weak metrics}

We consider a hyperbolic structure $X$ on $S$. 
We use the same letter $X$ to denote the corresponding element in $\mathcal{T}(S)$. 
Likewise, we take an element $\alpha\in\mathcal{B}(S)$, and we use the same letter $\alpha$ to denote its geodesic representative.

Let $\bar{\alpha}$ be the image of $\alpha$ by the canonical involution
of $S^d$ and let $\alpha^d = \alpha \cup \bar{\alpha}$.  

\begin{lemma}
\label{le:d} 
The curve $\alpha^d$ is a symmetric simple closed geodesic in $S^d$, and we have $l_{X^d}(\alpha^d) = 2l_X(\alpha)$.
\end{lemma}

\begin{proof} 
This simply follows from the fact that $\alpha^d$ intersects the boundary of $S$ perpendicularly, and that it is invariant under the canonical involution of $S^d$.
\end{proof}

\begin{corollary}
\label{cor:ineq}
Given two hyperbolic structures $X$ and $Y$ on $S$, we have 
\[\sup_{\gamma\in\mathcal{C}(S)\cup \mathcal{B}(S)}\frac{l_{X}(\gamma)}{l_{Y}(\gamma)}
\leq\sup_{\gamma\in\mathcal{C}(S^d)}\frac {l_{X^d}(\gamma)}{l_{Y^d}(\gamma)}.\]
\end{corollary}

\begin{proof}
For any $\gamma\in\mathcal{C}(S)\cup \mathcal{B}(S)$, we have, from Lemma \ref{le:d}, 
\[\frac{l_{X}(\gamma)}{l_{Y}(\gamma)}=
\frac{l_{X^d}(\gamma^{d})}{l_{Y^d}(\gamma^{d})}
\leq\sup_{\gamma\in\mathcal{C}(S^d)}\frac{l_{X^d}(\gamma)}{l_{Y^d}(\gamma)}.\]
Taking the supremum in the left hand side, we obtain the desired inequality.
\end{proof}

We now proceed to show the inverse inequality.

We recall that the set of weighted simple closed curves on $S^d$ is dense in the space $\mathcal{ML}(S^d)$, 
and that the hyperbolic length function, defined on weighted simple closed geodesics, extends to a continuous function defined on the space $\mathcal{ML}(S^d)$. 
With this in mind, and by compactness of the space $\mathcal{PML}(S^d)$, there is a measured geodesic lamination $\lambda$ on $S^d$ which realizes the supremum
\[\log\sup_{\gamma\in \mathcal{C}(S^d)}
        \frac {l_{X^d}(\gamma)}{l_{Y^d}(\gamma)}.\]
We shall use the following result of
Thurston, for surfaces without boundary, in which $i(\lambda_1, \lambda_2)$ denotes the geometric intersection number between the measured laminations $\lambda_1$ and $\lambda_2$ (see \cite{Thurston-notes} and \cite{Thurston1998}).  

\begin{theorem}[\cite{Thurston1998}, Theorem 8.2]
\label{th:Thurston1}
Let $S$ be a surface of finite type without boundary and let $X$ and $Y$ be two hyperbolic structures on $S$. 
If two measured geodesic laminations $\lambda_1$ and $\lambda_2$ attain the supremum
\[
\sup_{\gamma\in \mathcal{MF}(S)}
        \frac {l_{X}(\gamma)}{l_{Y}(\gamma)},\] then $i(\lambda_1, \lambda_2)=0$.
\end{theorem}

Consider again two hyperbolic structures $X$ and $Y$ on our surface with boundary $S$, 
and the associated doubled hyperbolic structures, $X^d$ and $Y^d$  on the double $S^d$ of $S$.

\begin{lemma} 
\label{le:l}
There is a symmetric measured
geodesic lamination on $S^d$ which realizes the supremum
$$\log\sup_{\gamma\in \mathcal{ML}(S^d)}
        \frac {l_{X^d}(\gamma)}{l_{Y^d}(\gamma)}.$$
\end{lemma}

\begin{proof}
Let $\lambda\in\mathcal{ML}(S^d)$ be a measured geodesic
lamination realizing the supremum.
Let $\lambda'$ be the image of $\lambda$ by the canonical involution on $S^d$. Since
this involution is an isometry for $X^d$ and $Y^d$, $\lambda'$ is a measured geodesic lamination on $S^d$ which realizes the
supremum as well. By Thurston's theorem (Theorem \ref{th:Thurston1}), it follows that
$i(\lambda,\lambda')=0$. The union $\lambda\cup\lambda'$ is
therefore a measured geodesic lamination. It realizes the supremum
and it is invariant under the canonical
involution, that is, it is symmetric.
\end{proof}

\begin{remark}
Every component of a symmetric measured geodesic lamination which meets the fixed point locus
of the involution is, if it exists, a simple closed geodesic. Indeed,
such a component must intersect the fixed point locus perpendicularly, and
no component which is not a simple closed geodesic can intersect the fixed locus in this way, because of
the recurrence of leaves.
\end{remark}

\begin{proposition}
\label{weak1}
With the above notations, we have
\[
 \log\sup_{\gamma\in\mathcal{C}(S)\cup \mathcal{B}(S)}
        \frac {l_{X}(\gamma)}{l_{Y}(\gamma)}
=\log\sup_{\gamma\in \mathcal{C}(S^d)}
        \frac {l_{X^d}(\gamma)}{l_{Y^d}(\gamma)}.
\]
\end{proposition}

\begin{proof}
By Lemma \ref{le:l}, there exists a symmetric measured lamination
$\lambda$ on $S^d$ such that
\[ \log\sup_{\gamma\in \mathcal{C}(S^d)}
        \frac {l_{X^d}(\gamma)}{l_{Y^d}(\gamma)}= \log \frac{l_{X^d} (\lambda)}{l_{Y^d} (\lambda)}.\]
Let $\alpha$ be any component of $\lambda\cap S$ (for the natural inclusion $S\subset S^d$). Then
$\alpha$ is either an arc joining two (possibly equal) boundary components of $S$, or a measured geodesic sub-lamination of $\lambda$
contained in the interior of $S$. By symmetry (or by Lemma
\ref{le:d}) and since $\alpha\subset\lambda$, the ratio of lengths
$\displaystyle \frac {l_{X}(\alpha)}{l_{Y}(\alpha)}$ is equal to the supremum
over $\gamma\in \mathcal{C}(S^{d})$ of the ratios of lengths $\displaystyle \frac {l_{X^d}(\gamma)}{l_{Y^d}(\gamma)}$. 
 This shows
that 
\[\sup_{\gamma\in\mathcal{C}(S)\cup \mathcal{B}(S)}\frac
{l_{X}(\gamma)}{l_{Y}(\gamma)} \geq\sup_{\gamma\in
\mathcal{C}(S^d)}\frac {l_{X^d}(\gamma)}{l_{Y^d}(\gamma)}.\]
The reverse inequality is given in Corollary \ref{cor:ineq}.
\end{proof}

\begin{corollary}
\label{co:double}
Let $S$ be a surface of topologically finite type.
For $X, Y$ in $\mathcal {T}(S)$,
$$
d(X,Y)=d(X^d,Y^d),\quad \overline{d}(X,Y)=\overline{d}(X^d,Y^d).
$$
\end{corollary}
  
\begin{proposition}
Let $S$ be a surface of topologically finite type.
Then $d$ and $\overline{d}$ are asymmetric weak metrics on $\mathcal{T}(S)$.
\end{proposition}
 
\begin{proof}
The triangle inequality is easily satisfied, and the separation property follows from the separation property of the metrics $d$ and $\overline{d}$ for surfaces without boundary
(which is Theorem 3.1 of Thurston \cite{Thurston1998}).
This proves that $d$ and $\overline{d}$ are weak metrics.
Thurston's example (\cite{Thurston1998}, \S 2, p.5) that shows the asymmetry of the weak metrics $d$ and $\overline{d}$ on $S^d$ can clearly be made symmetric with respect to the canonical involution on $S^d$ 
and therefore shows the asymmetry of the weak metrics $d$ and $\overline{d}$.
\end{proof}

We reformulate Corollary \ref{co:double} as follows:
 
\begin{corollary}
\label{co:double2}
The natural inclusion $\mathcal{T}(S)\to\mathcal{T}(S^d)$ given by doubling the hyperbolic structures 
is an isometry for the weak metric $d$ on $\mathcal{T}(S)$ and $d$ on $\mathcal{T}(S^d)$ 
(respectively for $\overline{d}$ on $\mathcal{T}(S)$ and $\overline{d}$ on $\mathcal{T}(S^d)$).
\end{corollary}

\begin{remark}
There is a natural projection $\mathcal{T}(S^{d})\to\mathcal{T}(S)$ which 
is the left inverse of the above natural inclusion.
\end{remark}

\subsection{Another expression for the weak metrics $d$ and $\overline{d}$}

We shall show that the weak metrics $d$ and $\overline{d}$ considered before can be expressed 
using suprema over the set $\mathcal{B}$ only.

\begin{proposition}
\label{another_expr}
For all $X,Y\in\mathcal{T}(S)$, one has 
$$
d(X,Y)=\log\sup_{\gamma\in\mathcal{B}}\frac{l_{Y}(\gamma)}{l_{X}(\gamma)},
$$
$$
\overline{d}(X,Y)=\log\sup_{\gamma\in\mathcal{B}}\frac{l_{X}(\gamma)}{l_{Y}(\gamma)}.
$$
\end{proposition}

We shall use the following technical lemma:

\begin{lemma}
\label{lemma:sym}
Let $\beta$ be a component of $\partial S$ and let $\alpha$
be a measured geodesic lamination on $S$, with $\alpha^d$ being its double (that is, $\alpha^d$ is the union of 
$\alpha$ and its image by the canonical involution). 
Then, there is a sequence of symmetric simple closed curves on $S^d$ converging to $\alpha^d$ in the topology of  
$\mathcal{PML}(S^d)$ such that each of these simple closed curves intersects essentially and in exactly two points the image of $\beta$ in $S^d$, 
and intersects no other component of the image of $\partial S$ in $S^d$ than $\beta$ itself.
\end{lemma}

\begin{proof}
By the density of the set of weighted simple closed geodesics in $\mathcal{ML}(S)$, 
we assume without loss of generality that $\alpha$ is a simple closed geodesic.
Thus, $\alpha^d$ is a symmetric measured geodesic lamination with two components, each of them being a simple closed curve ($\alpha^d$ is a multicurve).

We construct a train track $\tau$ on $S^d$ carrying $\alpha^d$ as follows.
Consider the symmetric multicurve $\alpha^d$.
Take two points on each component of $\alpha^d$ so as to get two pairs of symmetric points.
Connect these two pairs of points with a graph of the form ``$>$--$<$" so that the central branch, $e$, intersects the curve $\beta$
in one point in its interior and so that the result is a symmetric train track containing $\alpha^{d}$, as in Figure \ref{trucfig1}.  
Add two other branches, each connecting a component of $\alpha^d$ to itself, in a symmetric fashion,
so as to obtain a symmetric recurrent train track $\tau$, as shown in Figure \ref{trucfig1}.

Consider the following weights on the branches of $\tau$.
Put a weight 2 on the central branch $e$, and split this weight into unit weights on the four adjacent branches.
Put an integral weight $n\geq1$ on the two large branches of $\tau$ contained in $\alpha^{d}$, as shown in Figure \ref{trucfig1}.
These weights determine all the weights on the other branches of $\tau$.

It is easily seen that, for each $n$, the multicurve $\alpha^d_n$ obtained from these integral weights is connected and intersects essentially the image of $\beta$ in $S$ in exactly two points and that it intersects no other components of $\partial S$.
Furthermore, it is symmetric since $\alpha^{d}_{n}$ and its image by the involution determine the same weights
on the symmetric train track $\tau$.

Clearly, the sequence of weighted simple closed curves $(\frac{1}{n}\alpha^{d}_{n})$ converges to $\alpha^d$ in $\mathcal{ML}(S^{d})$ as desired.  

This concludes the proof.
\end{proof}

\begin{figure}[!hbp]
\centering
\psfrag{n}{$n$}
\psfrag{1}{$1$}
\psfrag{2}{$2$}
\psfrag{b}{$\beta$}
\includegraphics[width=.4\linewidth]{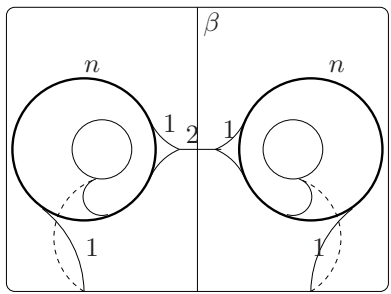}
\caption{\small{An example of a train track $\tau$ needed in the proof 
of Lemma
\ref{lemma:sym}. The symmetric multicurve $\alpha^{d}$ is represented in thick lines.
The weights for $\alpha_n^d$ on the various branches are indicated.}}
\label{trucfig1}
\end{figure}

\begin{proof}[Proof of Proposition \ref{another_expr}]
We only prove the proposition for the metric $d$ since the proof for the other metric will obviously be the same.
Let $\alpha$ be a simple closed curve on $S$.
Consider its double $\alpha^{d}$ in $S^d$. We take the sequence of symmetric weighted simple closed curves $(\gamma^d_{n})$ provided by Lemma \ref{lemma:sym} such that
$\lim_{n\to\infty}\gamma^d_{n}=\alpha^d$ in $\mathcal{ML}(S)$.  Then, for any two hyperbolic structures $X^d,Y^d$ on $S^d$ (which, for our needs, we can assume to be doubles of hyperbolic structures $X$ and $Y$ on $S$), we have a sequence of symmetric simple closed curves $(\gamma^d_{n})$ which, as elements of  $\mathcal{ML}(S)$, such that
$$
\Big{|}\frac{l_{Y^d}(\gamma^d_{n})}{l_{X^d}(\gamma^d_{n})}-
\frac{l_{Y^d}(\alpha^d)}{l_{X^d}(\alpha^d)}\Big{|}\to0.
$$
Note that we can assume that the curves  $(\gamma^d_{n})$ of  $\mathcal{ML}(S)$ are equipped with the counting measures, since their lengths in the above formula appear in quotients.

From the properties of the intersection with $\partial S$ of the curves $\gamma_n^d$ obtained from Lemma \ref{lemma:sym}, and taking now, for each $n$, the intersection $\gamma_n$ of this curve with $S$, we conclude that for any simple closed curve $\alpha$ in $S$ there exists a sequence of connected arcs $(\gamma_{n})$ in $\mathcal{B}$
such that
$$
\Big{|}\frac{l_{Y}(\gamma_{n})}{l_{X}(\gamma_{n})}-
\frac{l_{Y}(\alpha)}{l_{X}(\alpha)}\Big{|}\to0.
$$
  
This gives
$$
\sup_{\alpha\in\mathcal{B}\cup\mathcal{C}}\frac{l_{Y}(\alpha)}{l_{X}(\alpha)}
=
\sup_{\gamma\in\mathcal{B}}\frac{l_{Y}(\gamma)}{l_{X}(\gamma)},
$$
which concludes the proof.
\end{proof}

\section{Length spectrum metrics on $\mathcal{T}(S)$}
\label{section:length}

In this section,  $S$ is as before a surface of topologically finite type. We consider the following symmetrization of the weak asymmetric metrics that
we studied in the previous section.

$$\delta_L(X,Y) = \log\max \big\{\sup_{\gamma \in \mathcal{B}\cup\mathcal{C}}\frac{l_X(\gamma)}{l_Y(\gamma)},
\sup_{\gamma \in \mathcal {B}\cup \mathcal{C}}\frac{l_Y(\gamma)}{l_X(\gamma)}\big\}.$$ 
 
By Proposition \ref{another_expr}, we can also express this metric by
$$
 \delta_{L}(X,Y)=\log\max\big\{
 \sup_{\gamma\in\mathcal{B}}\frac{l_{Y}(\gamma)}{l_{X}(\gamma)},
 \sup_{\gamma\in\mathcal{B}}\frac{l_{X}(\gamma)}{l_{Y}(\gamma)}
 \big\}.
 $$
 
We also consider the following symmetric function on the product $\mathcal {T}(S)\times \mathcal {T}(S)$:

$$d_L(X,Y) = \log\max \big\{\sup_{\gamma \in \mathcal {C}}\frac{l_X(\gamma)}{l_Y(\gamma)},
\sup_{\gamma \in \mathcal
{C}}\frac{l_Y(\alpha)}{l_X(\alpha)}\big\}.$$

For surfaces of topologically finite type without bundary, $\delta_L$ and $d_L$ coincide, and they define the so-called {\it length spectrum metric} on Teichm\"uller space, which was originally defined by  Sorvali \cite{Sorvali},  and which has been studied by several authors, see e.g. \cite {Liu2001} and \cite{CR}.
From Corollary \ref{co:double}, we immediately deduce the following

\begin{corollary}\label{co:double2}
For $X, Y$ in $\mathcal{T}(S)$,
$$
\delta_L (X,Y)=\delta_L (X^d, Y^d).
$$
\end{corollary}

We shall prove that that $d_L$ is a metric. 
This metric $d_L$ has been studied by several authors for surfaces without boundary, 
and in that context it is usually called the {\it length-spectrum metric} on $\mathcal {T}(S)$  
(see e.g. \cite{Liu1999} and  \cite{Liu2001}).
Before proving that $d_L$ is a metric, we shall explain why we did not consider the asymmetric version of the length-spectrum metric. 

We then shall compare $\delta_{L}$ with $d_L$ in the ``relative-thick part" (to be defined below) of Teichm\"uller space.

\subsection{Necessity of symmetrization}

We consider the following non-sym\-metric version of $d_L$:
 \[K(X,Y)= \sup_{\gamma \in \mathcal {C}(S)}\frac{l_X(\gamma)}{l_Y(\gamma)}.\] 

The function $K$ is not a weak metric on $\mathcal{T}(S)$ in general. 
It can take negative values, as can easily be seen by considering a pair of pants (a sphere with three boundary components), 
equipped with hyperbolic metrics $X$ and $Y$ such that the lengths of the three boundary components for the metric 
$X$ are all strictly smaller than the corresponding lengths for the metric $Y$. 
In this case $K(X,Y)<0$.

Actually, one can ask whether this example of the pair of pants can be generalized.
In other words, given a hyperbolic surface with boundary, one can ask whether there exists another hyperbolic metric
on that surface for which the lengths of all simple closed geodesics is strictly decreased by a uniformly bounded amount.
The following result gives a partial answer to that question.

\begin{theorem}[Parlier \cite{Parlier}]
Let $S$ be a surface of topologically finite type with non-empty boundary.
Let $X$ be a hyperbolic structure on $S$. 
Then there exists a hyperbolic structure $Y$ on $S$ such that $K(X,Y)\leq0$. 
\end{theorem}

Indeed, Parlier shows that for any hyperbolic structure $X$ on any surface of topologically finite type with non-empty boundary, we can find a hyperbolic structure $Y$ on the same surface such that for any element $\gamma$ in $\mathcal{C}(S)$, we have $\displaystyle \frac{l_X(\gamma)}{l_Y(\gamma)}<1$ (which only gives $K(X,Y)\leq0$).

\subsection{Length spectrum metrics}

\begin{proposition} 
The functions  $d_L$ and $\delta_L$ are metrics on the Teichm\"uller space $\mathcal{T}(S)$.
\end{proposition}

\begin{proof}  
The fact that these functions satisfy the triangle inequality is immediate. 
The point to prove is that they are nonnegative and separate points. 
 
The function  $\delta_L$ is a symmetrization of the weak metric $d$ (or $\overline{d}$) which we considered in \S \ref{section:weak}, 
therefore it separates points, and it is a metric. 
This fact will also follow from the property $d_L(X,Y)\leq\delta_L(X,Y)$, 
and therefore it suffices to show that the function $d_L$ is nonnegative and separates points.

Assume that $d_L(X,Y)=0$.
This means that $K(X,Y)=K(Y,X)=1$. 
Equivalently, this means that the length of
any simple closed curve with respect to the hyperbolic structure $X$ is equal to the
length of that curve with respect to the hyperbolic structure $Y$. (Recall that boundary curves are included in $\mathcal {C}(S)$.)
By a well-known result (see e.g. \cite{FLP}), this implies that $X$ coincides with $Y$, as elements of Teichm\"uller space.

Now let us show that the $d_L$-distance between $X$ and $Y$ is nonnegative.

If $K(X,Y)>1$ or $K(Y,X)>1$ then the distance between $X$ and $Y$ is positive.

If $K(X,Y)\leq1$ and $K(Y,X)\leq 1$, then $K(X,Y)=K(Y,X)=1$.
This follows from the fact that if $K(X,Y)<1$, that is, if all simple closed
curves are strictly contracted from $X$ to $Y$, then $K(Y,X)>1$.
But we already noted that the equalities $K(X,Y)=K(Y,X)=1$ imply that $X=Y$.
\end{proof}

\subsection{Comparison between length spectrum metrics in the relative thick part of $\mathcal{T}(S)$}

In this section, $S$ is as before a surface of finite type with negative Euler-Poincar\'e characteristic and non-empty boundary.

The following is a well-established definition (see e.g. \cite{CR}).

Given $\epsilon > 0$, the {\it $\epsilon$-thick part of Teichm\"{u}ller space} is the set of $X\in\mathcal {T}(S)$ such that for any $\alpha\in\mathcal{C}(S)$
the hyperbolic length $l_X(\alpha)$ is not less than $\epsilon$. 

We now introduce the following terminology.

For $\epsilon >0$ and $\varepsilon_0\geq \epsilon$, the {\it $\varepsilon_0$-relative} $\epsilon$-thick part of Teichm\"{u}ller space is the subset of the $\epsilon$-thick part of Teichm\"{u}ller space
in which the length of each boundary component of $S$ is bounded from above by the constant $\varepsilon_{0}$.

Our first goal is to give a comparison between $d_L$ and $\delta_L$ on the $\varepsilon_0$-relative $\epsilon$-thick part of Teichm\"{u}ller space. 
This will be based on the following lemma, which, together with its proof, is analogous to Lemma 3.6 of Choi and Rafi \cite{CR} and its proof. 
The main tool used in this lemma is the technique of ``replacing an arc by a loop'', which is also the main tool in Choi and Rafi's lemma, 
and this technique was initiated by Minsky \cite{Minsky}. 
We prefer to give complete proofs rather than ideas of proofs, at the expense of repeating some of Choi and Rafi's arguments.

Given two functions $f$ and $g$ of a variable $t$, we shall use the notation `` $f \asymp g$ '' to express the fact that the ratio
$\displaystyle \frac{f(t)}{g(t)}$ is bounded above and below by some positive constants, with $t$ being in some domain which will be specified. 
Note that although we are following closely the arguments of Choi and Rafi in \cite{CR}, we are using the sign $\asymp$ differently.

 \begin{lemma}\label{Choi}

To any arc $\beta$  on $S$, we can associate a simple closed curve $\alpha$ on $S$ in such a way that the following holds
\begin{equation}\label{eq:beta}
l_X(\beta) \asymp l_X(\alpha),
\end{equation}
in which the variable is $X$, and where $X$ varies in the $\varepsilon_0$-relative $\epsilon$-thick part of Teichm\"uller space. 
(Thus, the multiplicative constants for $\asymp$ depend on $\epsilon$ and $\varepsilon_0$.)
 \end{lemma}
\begin{proof}

In all this proof, we consider a fixed arc  $\beta$ representing an element of $\mathcal {B}$.

  We first assume that $S$ is homeomorphic to a pair of pants, that is, to a surface with three holes, a hole being either a boundary component or a puncture of $S$. 
  
From the hypothesis, the lengths of the boundary components of this pair of pants are bounded from below and from above. 
  
There are several cases to consider, depending on whether the pair of pants $S$ has $1,2$ or $3$ boundary components (the other holes being punctures). 
In each case, there are only finitely many homotopy classes of arcs $\beta$, and there are distinctions between the cases where $\beta$ joins
one or two distinct boundary components. 
We show by trigonometric formulas that the lengths of these arcs are bounded from below and from above, in terms of $\epsilon$ and $\varepsilon_0$. 
As we shall see, the estimates are simple, and they give an idea of what happens in the general case.

We will have, in each case, $l_X(\beta) \asymp 1$ and $l_X(\alpha) \asymp 1$, which will show that (\ref{eq:beta}) holds. 
The cases are represented in Figure \ref{subcases1} (i) to (v). We analyze separately each of these cases.

\medskip
\noindent {\bf Case (i)} The pair of pants $S$ has 3 boundary components, and $\beta$ joins two distinct bounday components, $\gamma$ and $\gamma'$. 

Let $\gamma'''$ be the third component. 
We set $b=l_X(\beta)$, $a=l_X(\gamma)/2$, $a'=l_X(\gamma')/2$ and $a''=l_X(\gamma'')/2$. 
From Formula (\ref{eq:hexagon1}) for right-angled hexagons, we have
\[\cosh b=\frac{\cosh a''+\cosh a \cosh a'}{\sinh a\sinh a'}.\]
From the hypothesis, we have $a\asymp 1$, $a'\asymp 1$ and $a''\asymp 1$. 
Therefore, $b\asymp 1$, as required.

\medskip
\noindent {\bf Case (ii)} The pair of pants $S$ has 3 boundary components, and the two endpoints of $\beta$ are on one boundary component, which we call $\gamma$.

Let $\gamma'$ and $\gamma''$ be the other two boundary components of $S$. 
We consider the three arcs $\delta,\delta'$ and $\delta''$ joining pairwise the three boundary components of $S$ perpendicularly. 
The arc $\beta$ intersects perpendicularly one of these arcs, and it divides $S$ into two (in general non-isometric) right-angled hexagons. 
We consider one of these hexagons, say, the one containing the boundary component $\gamma'$ (and the arc $\delta''$). 
This hexagon is divided by the are $\delta''$ and by (part of) $\delta$ into two isometric right-angled pentagons. 
Let $b=l_X(\beta)/2$, $a'=l_X(\gamma')/2$ and $d''=l_X(\delta'')$. 
Then, Formula (\ref{eq:pentagon1}) for right-angled pentagons gives:
\[\cosh b=\sinh a'\sinh d''.\]
Again, from the hypothesis, we have $a'\asymp 1$, and from Case (i) above, we have $d''\asymp 1$. This gives $b\asymp 1$.

\medskip
\noindent  {\bf Case (iii)} The pair of pants $S$ has a unique boundary component $\gamma$. 

In that case, the isometry type of this surface is completely determined by the length of $\gamma$. There is only one arc $\beta$ to consider, and it joins $\gamma$ to itself. 
We consider the bi-infinite geodesic line joining the two punctures of $S$, 
and a geodesic ray starting perpendicularly on the boundary curve $\gamma$ and converging to one of these punctures. 
These lines are represented in the middle figure of the case labelled (iii) in Figure \ref{subcases1}. 
We set $a=l_X(\gamma)/4$ and $b=l_X(\beta)/2$. 
In the quadrilateral with three right angles and one zero angle (at infinity) represented to the right of this figure, we have (see \cite{Fenchel} p. 89)
\begin{equation}\label{eq: cusp}
\sinh a\sinh b=1.
\end{equation}
Again, since  $a\asymp 1$, we deduce  $b\asymp 1$.

\medskip
\noindent  {\bf Case (iv)} The pair of pants  $S$ has one puncture and 2 boundary components, 
and the arc $\beta$ has its endpoints on distinct boundary components, $\gamma$ and $\gamma'$.  

In that case, we consider the two geodesic rays starting perpendicularly at the two boundary components of $S$ and converging to the puncture. 
Cutting the surface along these two rays and the arc $\beta$, we obtain two isometric pentagons with four right angles and one zero angle (at infinity). 
We consider one of these pentagons and we set $a'=l_X(\gamma')/2$, $a=l_X(\gamma)/2$ and $b=l_X(\beta)$. 
We now consider the infinite ray in that pentagon that starts perpendicularly at the edge labelled $b$ and converges to the cusp, as shown Figure \ref{subcases1} (iv). 
The edge of length $b$ of the pentagon is divided into two segments, of lengths $b_1$ and $b_2$ satisfying $b_1+b_2=b$. 
By Formula (\ref{eq: cusp}) used above in Case (iii), and since $a\asymp 1$ and $a'\asymp 1$, we deduce that $b_1\asymp 1$ and $b_2\asymp 1$, therefore $b\asymp 1$.

\medskip
\noindent  {\bf Case (v)} The pair of pants $S$ has one puncture and 2 boundary components, and the arc $\beta$ has its endpoints on the  boundary component, which we call $\gamma$.  

In that case, $\beta$ divides $S$ into two components, one of which contains the other boundary curve of $S$, which we call $\gamma'$. 
The arc of minimal length, $\beta'$, joining $\gamma$ to $\gamma'$, 
together with an arc contained in the geodesic ray starting perpendicularly at $\gamma'$ and ending at the puncture, 
cut this component into two isometric right-angled pentagons. 
Let $a'=l_X(\gamma')/2$, $b'=l_X(\beta')$ and $b=l_X(\beta)/2$. 
From Case (iv) considered above, we have $b'\asymp 1$.
Since $a'\asymp 1$, Formula (\ref{pentagon1}) for right-angled pentagons gives again $b\asymp 1$.

\medskip

This concludes the proof of the lemma in the case where $S$ is a pair of pants.

Thus, in the case where $S$ is a pair of pants, we have an inequality
\begin{equation}\label{eq:pants}
\sup_{\gamma \in \mathcal{B}\cup \mathcal{C}}\frac{l_X(\gamma)}{l_Y(\gamma)} \leq  C \sup_{\alpha \in \mathcal
{C}}\frac{l_X(\alpha)}{l_Y(\alpha)}  
\end{equation}
where $C$ is a constant depending on $\epsilon$ and
$\varepsilon_0$.\\

\begin{figure}[!hbp]
\psfrag{1}{(i)}
\psfrag{2}{(ii)}
\psfrag{3}{(iii)}
\psfrag{4}{(iv)}
\psfrag{5}{(v)}
\psfrag{a}{$a$}
\psfrag{a1}{$a$}
\psfrag{a2}{$a'$}
\psfrag{a3}{$a''$}
\psfrag{b1}{$b$}
\psfrag{b2}{$b'$}
\psfrag{b3}{$b''$}
\psfrag{b}{$\beta$}
\psfrag{d}{$d''$}
\psfrag{d1}{$\delta$}
\psfrag{d2}{$\delta'$}
\psfrag{d3}{$\delta''$}
\psfrag{g1}{$\gamma$}
\psfrag{g2}{$\gamma'$}
\psfrag{g3}{$\gamma''$}
\centering
\includegraphics[width=\linewidth]{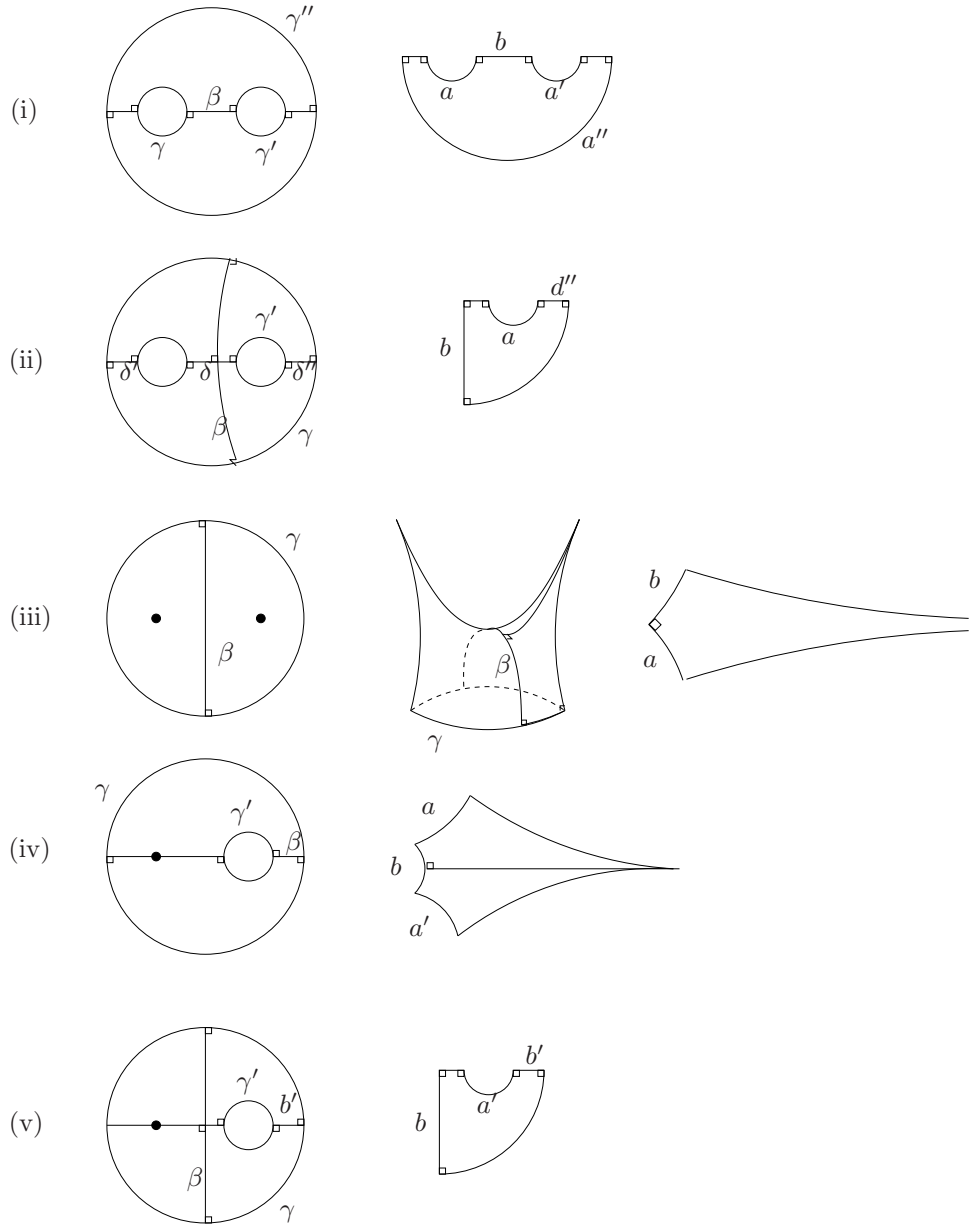}
\caption{\small{The various subases in Lemma \ref{Choi}, in the case where $S$ is a pair of pants (with boundary components or with punctures), and $\beta$ an arc joining boundary components.}}
\label{subcases1}
\end{figure}

Now we consider the case where $S$ is not a pair of pants.

As in the preceding case, we also show that we can associate to $\beta$ a simple closed curve $\alpha$ on $S$ such that 
$l_X(\beta) \asymp l_X(\alpha).$
 
Of course, in contrast to the preceding case, there are infinitely many homotopy classes of arcs $\beta$ to consider now.  
The idea of getting the closed curve $\alpha$ is similar to the one in the case where $S$ is a pair of pants, 
and instead of working in a pair of pants $S$ 
we work in a pair of pants that is the regular neighborhood of the union of $\beta$ with the boundary components of $S$ that it intersects. 
The curve $\alpha$ associated to the arc $\beta$ will be one of the boundary curves of this pair of pants, and this curve is no more necessarily a boundary curve of $S$; 
its geodesic length is bounded from below, but not necessarily from above. 
We now give the details.

Without loss of generality, we assume that $\beta$ is geodesic for the hyperbolic metric $X$. 

If the two endpoints of $\beta$ lie on distinct components $\gamma, \gamma'$ of
$\partial X$, then the boundary of a regular neighborhood of $\beta
\cup \gamma \cup \gamma'$ in $X$ consists of a single simple closed curve $\alpha$. 
Since $S$ is not a pair of pants, $\alpha$ is not peripheral and therefore
$\alpha \in \mathcal {C}(S)$. 
We take $\alpha$ to be the geodesic representative of its homotopy class with respect to the hyperbolic structure $X$.

If both endpoints of $\beta$ lie on the same component $\gamma$ of $\partial X$, 
then the boundary of a regular neighborhood of $\beta\cup \gamma$ consists of two disjoint simple closed curves. 
Both curves cannot be peripheral since $S$ is not a pair of pants. 
We take $\alpha$ to be in the homotopy class of a non-peripheral curve. 
In the case where both curves are non-peripheral, we take $\alpha$ to be in the homotopy class whose $X$-geodesic has greater length 
(and if both lengths are equal, we choose any one of the two curves). 
Again, we take $\alpha$ to be the geodesic representative in its homotopy class.

For this choice of the curve $\alpha$, we claim that (\ref{eq:beta}) holds.

We prove this claim by elementary hyperbolic trigonometry. 
In what follows, a ``geodesic pair of pants" means a pair of pants with geodesic boundary embedded in a hyperbolic surface.
We distinguish two cases.

\noindent \textbf{Case 1.}\  This is the case where $\gamma \neq\gamma'$. 
Let $P$ be the geodesic pair of pants in $X$ with boundary components $\gamma, \gamma', \alpha$ 
and consider one of the two canonical right-angled geodesic hexagons that divide $P$. 
Let $a =l_X(\gamma)/2, a' =  l_X(\gamma')/2, c = l_X(\alpha)/ 2$ and $b = l_X(\beta)$. 
From the formula for right-angled hexagons (Formula (\ref{eq:hexagon1}) above), we have
\begin{equation}\label{eq:hexa}
\cosh c + \cosh a\cosh a' = \sinh a\sinh a'\cosh b.
\end{equation}

Since $a, a' \leq \frac{\varepsilon_0}{2}$ and from the fact that the continuous function $x\mapsto\displaystyle \frac{\sinh x}{x}$ 
has a finite limit as $x$ converges to zero, there is a constant
$C(\varepsilon_0)>0$ depending on $\varepsilon_0$ such that
\[\sinh a\leq C(\varepsilon_0)a   \hbox{ and } \sinh a'\leq
C(\varepsilon_0)a'.\]
 We shall also use the well-known inequalities
\[x< \sinh x   \hbox{ and }  \frac{e^x}{2}< \cosh x< e^x  \hbox{ for all } \ x> 0.\]

In particular, we have $\cosh x  \asymp e^x$ for all $x>0$.

By assumption, we have $\displaystyle a, a' \geq \frac{\epsilon}{2}$. 
Therefore,
\[\sinh a\sinh a'\cosh b \geq a\cdot a'\cdot \frac{e^b}{2} \geq \frac{\epsilon^2e^b}{8}\]
and 
\[\sinh a\sinh a'\cosh b \leq C(\varepsilon_0)^2 a\cdot a' \cdot e^b \leq \frac{C(\varepsilon_0)^2\varepsilon_0^2e^b}{4}.\]
Hence,
\begin{equation}\label{eq5}
\sinh a\sinh a'\cosh b \asymp e^b.
\end{equation}

On the other hand, we have
$$\cosh a\cosh a' \leq \cosh(a + a') \leq \cosh \varepsilon_0,$$
$$\cosh c \geq \cosh \frac{\epsilon}{2}>0.$$
We get
\begin{equation}\label{eq6}
\cosh c \leq \cosh c + \cosh a\cosh a'\leq (1 + \frac{\cosh \varepsilon_0}{\cosh \frac{\epsilon}{2}})\cosh c,
\end{equation}
whence
\begin{equation}\label{eq7}
e^c \asymp \cosh c \asymp \cosh c + \cosh a\cosh a'\hbox{ for } c\geq 0.
\end{equation}
Now considering $c$ and $b$ as functions of hyperbolic structures on $S$, 
it follows from  (\ref{eq5}),(\ref{eq7}) and (\ref{eq:hexa}) that
\[e^c \asymp e^b,\] 
for hyperbolic structures varying in the $\varepsilon_0$-relative $\epsilon$-thick part of Teichm\"uller space $\mathcal{T}(S)$.

Thus, there is a constant $K>0$ depending on $\epsilon$ and $\varepsilon_0$ such that
\begin{equation}\label{eq:K}
|c-b| = | l_X(\alpha)/2 - l_X(\beta)| \leq K.
\end{equation}
Since $l_X(\alpha) \geq \epsilon$, we have
\begin{equation}\label{eq:KK}
 l_X(\beta) \leq l_X(\alpha)/2 + K \leq (\frac{1}{2} +
\frac{K}{\epsilon}) l_X(\alpha).
\end{equation}

We now distinguish two sub-cases: 

1) If $l_X(\beta)\geq 1$, then we have from (\ref{eq:K}):
\[l_X(\alpha) \leq 2l_X(\beta) + 2K \leq 2(1 + K) l_X(\beta).\]
Combined with (\ref {eq:KK}), this gives, in this sub-case, $l_X(\alpha) \asymp l_X(\beta)$. 

2) If $l_X(\beta)< 1$, then from the definition of $\alpha$, we have
$$l_X(\alpha)\leq 2l_X(\beta) + l_X(\gamma) + l_X(\gamma')
\leq 2l_X(\beta) + 2\varepsilon_0< 2 + 2\varepsilon_0.$$ 
As a result, $l_X(\alpha)$ is bounded below and above. 
Now we use Formula (\ref{eq:hexagon1}) again: since $a, a'$ and $c =  l_X(\alpha)/2$ are bounded
below and above, $b = l_X(\beta)$ is bounded below and above.

Thus, in both sub-cases, we have $l_X(\alpha) \asymp l_X(\beta)$.

\noindent \textbf{Case 2.}\  Next we consider the case where $\gamma = \gamma'$.
Let $P$ be the geodesic pair of pants in $X$ with boundary
components $\gamma, \alpha$ and the geodesic representative of the
third simple closed curve, which we now call $\alpha'$, which is in the homotopy class of a regular neighborhood of $\beta \cup \gamma$. 
The curve $\alpha'$ may be peripheral, and in that case  we say that the geodesic is at the puncture and has length zero.
We take the canonical division of the geodesic pair of pants $P$ into two isometric right-angled hexagons,
when  $\alpha'$ is not peripheral, or two isometric pentagons with four right-angles and one zero angle, 
when $\alpha'$ is peripheral.  
By symmetry, the arc $\beta$ divides these two geodesic regions into four geodesic pieces,  
two of them being isometric right-angled pentagons with edges originally contained in $\alpha$
and the other two being isometric right-angled pentagons if $\alpha'$ is not peripheral, 
or isometric squares with three right-angles and one zero angle otherwise. 

\begin{figure}[!hbp]
\psfrag{a1}{$\alpha$}
\psfrag{a2}{$\alpha'$}
\psfrag{b}{$\beta$}
\psfrag{g}{$\gamma$}
\psfrag{q}{$Q$}
\centering
\includegraphics[width=.4\linewidth]{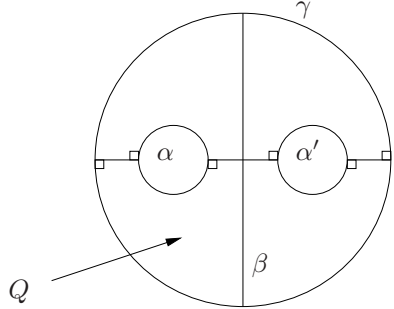}
\caption{\small{$Q$ is the pentagon referred to in the proof of Lemma \ref{Choi}.}}
\label{pentagon2}
\end{figure}

Let $Q$ be one of the two pentagons that have edges originally contained in $\alpha$ 
(Figure \ref{pentagon2}, in the case where $\alpha'$ is not peripheral).
Let $b = l_X(\beta)/2$, let $c = l_X(\alpha)/2$ and let $a$ be the length of the edge of $Q$ that arises from $\gamma$. 
From the formula for right-angled pentagons (Equation (\ref{eq:pentagon1}) above), we have
\begin{equation}\label{eq:penta}
 \cosh c = \sinh b\sinh a.
\end{equation}

It is clear that $a\leq l_X(\gamma)/2$.
By applying the pentagon formula (which is valid in the limiting case where $c$ has length zero) to the pentagon which together with $Q$ makes a hexagon of $P$, 
we see that our choice of $\alpha$ (of length not less than the length of $\alpha'$) implies that $a\geq l_X(\gamma)/4$.

Since $\epsilon \leq l_X(\gamma)\leq\varepsilon_0$, we have $\epsilon/4 \leq a \leq \varepsilon_0/2$.

Recall that there is a constant $C(\varepsilon_0)>0$ depending on $\varepsilon_0$
such that $a \leq \sinh a \leq C(\varepsilon_0)a $. 
We get
\begin{equation}
\sinh b \sinh a > \frac{e^b\cdot a}{2} > \frac{e^b \cdot \epsilon}{8}
\end{equation}
and
\begin{equation}
\sinh b \sinh a < e^b\cdot C(\varepsilon_0)\cdot a < \frac{C(\varepsilon_0)\cdot\varepsilon_0\cdot  e^b}{2}.
\end{equation}
The two preceding inequalities give $\sinh b \sinh a\asymp e^b$ and it follows from (\ref{eq:penta}) that 
\[e^c \asymp e^b.\]
Thus, there is a constant $K>0$ depending on $\epsilon$ and $\varepsilon_0$ such that
$$|c - b| = | l_X(\alpha)/2 - l_X(\beta)/2 | \leq K.$$
Note that $l_X(\alpha) \geq \epsilon$ and $l_X(\alpha) \leq
l_X(\beta) + l_X(\gamma) \leq l_X(\beta) + \varepsilon_0$, we can
use the same argument as in  Case 1 to conclude that
$l_X(\beta) \asymp l_X(\alpha)$.\\

The lemma is proved.
\end{proof}

From Lemma \ref{Choi}, we deduce the following
\begin{proposition}\label{prop:bi}

Let $0<\epsilon\leq\varepsilon_0$ be two positive numbers.
Then there is a positive constant $C>0$, depending only on $\epsilon$ and $\varepsilon_0$ and the type of $S$,
such that, for any $X,Y$ in the $\varepsilon_0$-relative $\epsilon$-thick part of $\mathcal{T}(S)$, 
we have
$$ \sup_{\alpha \in \mathcal
{C}}\frac{l_X(\alpha)}{l_Y(\alpha)}
\leq
\sup_{\alpha \in \mathcal{B}\cup\mathcal{C}}\frac{l_X(\alpha)}{l_Y(\alpha)}
\leq C \sup_{\alpha \in \mathcal
{C}}\frac{l_X(\alpha)}{l_Y(\alpha)}.$$
\end{proposition}

\begin{proof}

The first inequality is trivial, and the second one follows from Lemma \ref{Choi}.
\end{proof}

The following theorem follows then from the above proposition.

\begin{theorem}\label{th:com}
Let $S$ be a surface of topologically finite type.
Fix two positive constants $\epsilon\leq\varepsilon_0$.
Then there exists a constant $K$ depending on $\epsilon$, $\varepsilon_{0}$ and the topological type of $S$ such that, 
for any $X,Y$ in the $\varepsilon_{0}$-relative $\epsilon$-thick part of $\mathcal {T}(S)$, one has
\[d_L(X,Y) \leq \delta_L(X,Y) \leq d_L(X,Y) + K.\]
\end{theorem}

\begin{proof}
The first inequality is trivial, and it holds without the assumptions on the bounds on the geometry of the structures $X$ and $Y$. 
We prove the second one.

From Proposition \ref{prop:bi}, there exists a constant $C$ depending on $\epsilon$ and $\varepsilon_0$ such that
\[\sup_{\alpha \in \mathcal{C}(S)\cup \mathcal{B}(S)}\frac{l_X(\alpha)}{l_Y(\alpha)} \leq C \sup_{\alpha \in
\mathcal {C}(S)}\frac{l_X(\alpha)}{l_Y(\alpha)}\]
and
\[\sup_{\alpha \in  \mathcal{C}(S)\cup \mathcal{B}(S)}\frac{l_Y(\alpha)}{l_X(\alpha)} \leq C \sup_{\alpha \in
\mathcal {C}(S)}\frac{l_Y(\alpha)}{l_X(\alpha)}\]
By taking logarithms and summing the two equations, we get
$$\delta_L(X,Y) \leq d_L(X,Y) + \frac{\log C}{2}.$$
\end{proof}

Let us also note the following:

\begin{theorem}\label{th:com2}
Let $S$ be a surface of topologically finite type.
Fix two positive constants $\epsilon\leq\varepsilon_0$.
Then there exists a constant $K$ depending on $\epsilon$, $\varepsilon_{0}$ and the topological type of $S$ such that, 
for any $X,Y$ in the $\varepsilon_{0}$-relative $\epsilon$-thick part of $\mathcal {T}(S)$, we have
\[d(X,Y) \leq \delta_L(X,Y) \leq d(X,Y) + K\]
and 
\[\overline{d}(X,Y) \leq \delta_L(X,Y) \leq \overline{d}(X,Y) + K.\]
\end{theorem}

\begin{proof}
We only prove the first statement, since the second one follows with the same arguments.
The inequality $d(X,Y) \leq \delta_L(X,Y)$ follows from the definitions.
For the other inequality, we use the doubled surfaces. By Theorem B in Choi and Rafi \cite{CR},  there exists a constant $K$ depending
on $\epsilon, \ \varepsilon_0$ and the  topological type of $S$ such that
\[
\delta_L (X^d, Y^d)\le d(X^d, Y^d)+K.
\]
Using Corollary \ref{co:double} and Corollary \ref{co:double2}, this implies that $\delta_L(X,Y)\le d(X,Y)+K.$ This proves the required result.
\end{proof}

A natural question which arises after this, is what happens if we omit the condition on the metrics $X$ and $Y$ being in the $\varepsilon_{0}$-relative $\epsilon$-thick part of $\mathcal {T}(S)$. The following simple example shows that the second inequality in the statement of Theorem \ref{th:com2} cannot hold in general.

\begin{example}
For every $t\geq 1$, let $X_t$ be the hyperbolic structure on a pair of pants with three geodesic boundary components of equal length $t$. Then, it clear that for every $t\geq 1$, we have $d_T(X_1, X_t)=\log t$. Now as $t\to\infty$, the common length $l_t$ of the three seams joining pairwise the three boundary components of the pair of pants is of the order of $e^{-\frac{t}{2}}$. Indeed,
by Formula (\ref{eq:hexagon1}) for right-angles hexagons, we have $\cosh l_t=(\cosh^2 t +\cosh t)/\sinh^2 t$, which gives 
\[\cosh l_t-1=\frac{1+\cosh t}{\sinh^2 t}=\frac{1+\cosh t}{\cosh^2 t -1}= 
\frac{1}{\cosh t -1}\]
which implies $l_t^2\sim4e^{-t}$, that is, $l_t\sim 2e^{-\frac{t}{2}}$.
This gives $\delta_L(X,Y)\asymp\log e^{\frac{t}{2}}=\frac{t}{2}$.
\end{example}

\section{The topology induced by the metrics on $\mathcal{T}(S)$}

We also have the following

\begin{theorem}
Let $S$ be a surface of topologically finite type and 
let $(X_n)_{n\geq0}$ be a sequence of elements in  $\mathcal {T}(S)$.
Then,
$$\lim_{n\rightarrow \infty}d_L(X_n, X_0) = \infty
\ \textmd{if} \ \textmd{and} \ \textmd{only} \ \textmd{if} \
\lim_{n\rightarrow \infty}\delta_L(X_n, X_0) = \infty.$$
\end{theorem}

\begin{proof}
Since $d_L(X,Y) \leq \delta_L(X,Y)$, it is clear that $d_L(X_n, X_0)\to\infty$ implies
$\delta_L(X_n, X_0) \to\infty$.

Now assume that $\delta_L(X_n, X_0)\to\infty$. 
We show that $d_L(X_n, X_0) \to \infty$. 
Suppose not.
Then there is a constant $K$ and a subsequence $(X_{n_i})$ such that $d_L(X_{n_i}, X_0) \leq K$. 
From the definition of the length-spectrum metric, all the $X_{n_i}$, $i\geq1$,
and $X_0$ must lie in some $\epsilon$-thick part of $\mathcal{T}(S)$ and satisfy the condition that all the boundary components
are bounded above by some constant $\varepsilon_0$. 
By Theorem \ref{th:com}, $\delta_L(X_{n_i}, X_0)$ is bounded, which is a contradiction.
\end{proof}

Let us also note the following

\begin{theorem}Let $S$ be a surface of topologically finite type and
let $(X_n)_{n\ge 0}$ be a sequence of elements in $\mathcal{T}(S)$.
Then,
\[
 lim_{n\to\infty}d(X_n, X_0)=\infty\iff 
\ lim_{n\to\infty}\delta_L(X_n, X_0)=\infty\iff lim_{n\to\infty}\bar{d}(X_n, X_0)=\infty.\]
\end{theorem}
\begin{proof} We only prove the first equivalence, the proof of the
second one being similar.

As $d(X_n, X_0)\le\delta_L(X_n, X_0)$, we have 
\[\lim_{n\to\infty}d(X_n,
X_0)=\infty\Rightarrow \lim_{n\to\infty}\delta_L(X_n,
X_0)=\infty.\]

Now suppose that $\lim_{n\to\infty}\delta_L(X_n, X_0)=\infty$, Then
by Corollary \ref{co:double2},  we have $\lim_{n\to\infty}\delta_L(X^d_n, X^d_0)=\infty$.
By a result proved in Liu \cite{Liu2001} and Papadopoulos \& Th\'eret in 
  \cite{Papado-Th2007}, we have $\lim_{n\to\infty}d(X^d_n,
X^d_0)=\infty$. By Corollary \ref{co:double}, we obtain $\lim_{n\to\infty}d(X_n,
X_0)=\infty $.
\end{proof}

We now give a relation between  $\delta_L$ when $d_L$ in the small range.
 
 \begin{theorem}
Let $S$ be a surface of topologically finite type and  
let $(X_n)_{n\geq0}$  be a sequence of elements in  $\mathcal {T}(S)$.
Then,
\[\lim_{n\rightarrow \infty}d_L(X_n, X_0) = 0
\ \textmd{if} \ \textmd{and} \ \textmd{only} \ \textmd{if} \
\lim_{n\rightarrow \infty}\delta_L(X_n, X_0) = 0.\]

\end{theorem}

\begin{proof} That $\lim_{n\rightarrow \infty}\delta_L(X_n, X_0) = 0
\Rightarrow
\lim_{n\rightarrow \infty}d_L(X_n, X_0) = 0$ follows from the inequality $d_L\leq \delta_L$. We prove the converse.

Assume that $\lim_{n\rightarrow \infty}d_L(X_n, X_0) = 0$. Equivalently, we have
\[\sup_{\alpha\in\mathcal{C}(S)}\frac{l_{X_{0}}(\alpha)}{l_{X_{n}}(\alpha)}\to 1\hbox{ and }
\sup_{\alpha\in\mathcal{C}(S)}\frac{l_{X_{n}}(\alpha)}{l_{X_{0}}(\alpha)}\to 1 \hbox{ as $n\to\infty$}.\]

Using the fact that 
\[\inf_{\alpha\in\mathcal{C}(S)}\frac{l_{X_{n}}(\alpha)}{l_{X_{0}}(\alpha)}=
\left(\sup_{\alpha\in\mathcal{C}(S)}\frac{l_{X_{0}}(\alpha)}{l_{X_{n}}(\alpha)}\right)^{-1},\]
we have

 \begin{eqnarray*}
\lim_{n\rightarrow \infty}d_L(X_n, X_0) = 0&\iff &\displaystyle
\inf_{\alpha\in\mathcal{C}(S)}\frac{l_{X_{n}}(\alpha)}{l_{X_{0}}(\alpha)}\to 1, \ 
\sup_{\alpha\in\mathcal{C}(S)}\frac{l_{X_{n}}(\alpha)}{l_{X_{0}}(\alpha)}\to 1 \hbox{ as $n\to\infty$}\\
&\Rightarrow& l_{X_{n}}(\alpha)\to l_{X_{n}}(\alpha)\  \forall\alpha\in\mathcal{C}(S)\\
&\Rightarrow&X_n\to X_0 \hbox { in the usual topology of $\mathcal{T}(S)$}.
\end{eqnarray*}
 This implies that $X_n^d\to X_0^d$ as $n\to \infty$, using, for instance, Fenchel-Nielsen coordinates in the Teichm\"uller space of the doubled surface.
 
Thus, we have
\begin{equation}\label{eq:conv}d_L(X_n^d,X_0^d)\to 0 \hbox{ as $n\to\infty$},
\end{equation} where $d_L$ denotes here the length-spectrum distance in the double.
This last limit follows because the doubled surface is a surface without boundary
which implies, by a result in \cite{Papado-Th2007} (Corollary 6 p. 495), that 
\[X_n^d\to X_0^d\iff 
\sup_{\alpha\in\mathcal{C}(S^d)}\frac{l_{X_{0}^{d}}(\alpha)}{l_{X_{n}^{d}}(\alpha)}\to 1
\iff 
\sup_{\alpha\in\mathcal{C}(S^d)}\frac{l_{X_{n}^{d}}(\alpha)}{l_{X_{0}^{d}}(\alpha)}\to 1 \hbox{ as $n\to\infty$}.\]

We now deduce from (\ref{eq:conv}) that 
\begin{equation}\label{eq:conv1}
\inf_{\alpha\in\mathcal{C}(S^d)}\frac{l_{X_{n}^{d}}(\alpha)}{l_{X_{0}^{d}}(\alpha)}\to 1 \hbox{ and } 
\sup_{\alpha\in\mathcal{C}(S^d)}\frac{l_{X_{n}^{d}}(\alpha)}{l_{X_{0}^{d}}(\alpha)}\to 1 \hbox{ as $n\to\infty$.}
\end{equation}

Now we note that
\begin{equation}\label{eq:conv2}
\inf_{\alpha\in\mathcal{C}(S^d)}\frac{l_{X_{n}^{d}}(\alpha)}{l_{X_{0}^{d}}(\alpha)}\leq
\inf_{\alpha^d\in\mathcal{C}(S^d)}\frac{l_{X_{n}^{d}}(\alpha^d)}{l_{X_{0}^{d}}(\alpha^d)}\leq
\sup_{\alpha^d\in\mathcal{C}(S^d)}\frac{l_{X_{n}^{d}}(\alpha^d)}{l_{X_{0}^{d}}(\alpha^d)}\leq
\sup_{\alpha\in\mathcal{C}(S^d)}\frac{l_{X_{n}^{d}}(\alpha)}{l_{X_{0}^{d}}(\alpha)}
\end{equation}
where the notation $\alpha^d$ for a simple closed curve in $S^d$ means, as before, that we consider a symmetric curve with respect to the natural involution.

From  (\ref{eq:conv1}) and  (\ref{eq:conv2}) we now deduce that
 \begin{eqnarray*}
 \lim_{n\to\infty}d_L(X_n,X_0)=0&\Rightarrow&
 \inf_{\alpha^d\in\mathcal{C}(S^d)}\frac{l_{X_{n}^{d}}(\alpha^d)}{l_{X_{0}^{d}}(\alpha^d)}\to 1, 
\sup_{\alpha^d\in\mathcal{C}(S^d)}\frac{l_{X_{n}^{d}}(\alpha^d)}{l_{X_{0}^{d}}(\alpha^d)}\to 1 \hbox{ as $n\to\infty$}\\
&\Rightarrow &\displaystyle
\inf_{\alpha\in\mathcal{B}(S)}\frac{l_{X_{n}}(\alpha)}{l_{X_{0}}(\alpha)}\to 1, 
\sup_{\alpha\in\mathcal{B}(S)}\frac{l_{X_{n}}(\alpha)}{l_{X_{0}}(\alpha)}\to 1 \hbox{ as $n\to\infty$}.
\end{eqnarray*}

This shows that $\delta_L(X_n, X_0) = 0$ as $n\to\infty$, which completes the proof.

\end{proof}

We also have the following

\begin{theorem}Let $S$ be a surface of topologically finite type and
let $(X_n)_{n\ge 0}$ be a sequence of elements in $\mathcal{T}(S)$.
Then, we have
\[\lim_{n\to\infty}d(X_n, X_0)=0 \iff
\lim_{n\to\infty}\bar{d}(X_n, X_0)=0 \iff 
 \lim_{n\to\infty}\delta_L(X_n, X_0)=0.\]
\end{theorem}

\begin{proof} The first equivalence follows by taking doubles, and using again the results in \cite{Liu2001} and \cite{Papado-Th2007}. The second equivalence is obvious.
\end{proof}

\begin{corollary}
Let $S$ be a surface of topologically finite
type. 
Then $d$, $\overline{d}$, $d_L$ and $\delta_L$ induce the same topology on
$\mathcal {T}(S)$.
\end{corollary}

\end{document}